\documentclass{amsart}

\usepackage{amssymb}
\usepackage{amscd}
\usepackage{amsfonts}
\usepackage{version}

\numberwithin{equation}{section}
\theoremstyle{plain}
\newtheorem{theorem}[subsection]{Theorem}

\newtheorem{lemma}[subsection]{Lemma}

\newtheorem{definition}[subsection]{Definition}

\renewcommand{\leq}{\leqslant}
\renewcommand{\geq}{\geqslant}
\newsavebox{\proofbox}
\savebox{\proofbox}{\begin{picture}(7,7)  \put(0,0){\framebox(7,7){}}\end{picture}}

\newcommand\Z{\mathbb{Z}}

\newcommand\C{\mathbb{C}}

\newcommand\SL{\operatorname{SL}}

\newcommand\im{\operatorname{im}}

\newcommand\GL{\operatorname{GL}}

\newcommand\tr{\operatorname{tr}}
\newcommand\rss{\operatorname{rss}}

\newcommand\F{\mathbb{F}}

\newcommand\id{\operatorname{id}}
\def\proof{\noindent\textit{Proof. }}

\def\endproof{\hfill{\usebox{\proofbox}}\vspace{11pt}}

\begin{document}

\title{Linear Approximate Groups}
\author{Emmanuel Breuillard}
\address{Laboratoire de Math\'ematiques\\
B\^atiment 425, Universit\'e Paris Sud 11\\
91405 Orsay\\
FRANCE}
\email{emmanuel.breuillard@math.u-psud.fr}

\author{Ben Green}
\address{Centre for Mathematical Sciences\\
Wilberforce Road\\
Cambridge CB3 0WA\\
England }
\email{b.j.green@dpmms.cam.ac.uk}

\author{Terence Tao}
\address{Department of Mathematics, UCLA\\
405 Hilgard Ave\\
Los Angeles CA 90095\\
USA}
\email{tao@math.ucla.edu}


\begin{abstract} This is an informal announcement of results to be described and proved in detail in \cite{bgst-long}. We give various results on the structure of approximate subgroups in linear groups such as $\SL_n(k)$. For example, generalising a result of Helfgott (who handled the cases $n = 2$ and $3$), we show that any approximate subgroup of $\SL_n(\F_q)$ which generates the group must be either very small or else nearly all of $\SL_n(\F_q)$. The argument is valid for all Chevalley groups $G(\F_q)$. Extending work of Bourgain-Gamburd we also give some applications to expanders. \end{abstract}

\maketitle
\tableofcontents

\section{Introduction}

In this note we announce some new results on approximate subgroups in linear groups.
We begin by recalling the notion of an approximate group, first introduced (in the non-abelian setting) in \cite{tao-noncommutative}. See \cite{green-survey} for a more extensive motivating discussion.

\begin{definition}[Approximate groups] \label{def1.1} Let $K \geq 1$. A nonempty finite set $A$ in some ambient group $G$ is
called a \emph{$K$-approximate group} if
\begin{enumerate}
\item It is symmetric, i.e. if $a \in A$ then $a^{-1} \in A$, and the identity lies in $A$;
\item There is a symmetric subset $X \subseteq G$ with $|X| \leq K$ such that $A \cdot A \subseteq X \cdot A$, where $A \cdot A = \{a_1 a_2 : a_1, a_2 \in A\}$ is the product set of $A$ with itself.
\end{enumerate}
\end{definition}

Note in particular that a $1$-approximate group in $G$ is the same thing as a finite subgroup of $G$. For the rest of the paper we will assume that $K \geq 2$, and in this regime there are $K$-approximate groups which are not close to genuine groups; the simplest example is that of a geometric progression $\{g^n : |n| \leq N\}$, and there also exist higher-dimensional and nilpotent generalisations of this.  Again, \cite{green-survey} may be consulted for further discussion.

Many papers have been written in which the aim is to classify a certain class of approximate groups. For example, the Fre\u{\i}man-Ruzsa theorem \cite{ruz-freiman} provides a description of approximate subgroups of the integers $\Z$. ``Classification'' in this context must be interpreted quite loosely. The following notion of \emph{control}, first introduced in \cite{tao-solvable}, has proved very useful in this context.

\begin{definition}[Control] \label{def1.2}Suppose that $A$ and $B$ are two sets in some ambient group,
and that $K \geq 1$ is a parameter. We say that $A$ is $K$-controlled by $B$, or that $B$ $K$-
controls $A$, if $|B| \leq K|A|$ and there is some set $X$ in the ambient group with $|X | \leq K$
and such that $A \subseteq (X \cdot B) \cap (B \cdot X)$.
\end{definition}

Given this definition, one may describe the classification problem for approximate groups as follows: given some class $\mathcal{C}$ of approximate groups, find some smaller, more highly-structured, class of approximate groups $\mathcal{C}'$ such that every object in $\mathcal{C}$ is efficiently controlled by an object in $\mathcal{C}'$.\vspace{11pt}

\noindent\textsc{notation.} The letter $C$ always denotes an absolute constant, but different instances of the notation may refer to different constants. If $C$ depends on some other parameter (for example, if we are working in $\SL_n$, $C$ might need to depend on $n$) then we will indicate this dependence with subscripts. If $A$ is a finite set then $|A|$ denotes the cardinality of $A$. The letter $K$ is reserved for the positive real parameter appearing in the definition of approximate group. For non-negative quantities $X, Y$, we use $X \lesssim Y$ or $Y \gtrsim X$ to denote the estimate $X \leq K^C Y$, and $X \sim Y$ to denote the estimates $X \lesssim Y \lesssim X$. The symbol $p$ always denotes a prime number, and $\F_p$ denotes the field of order $p$. Finally, we use $A^k := \{ a_1 \ldots a_k: a_1,\ldots,a_k \in A\}$ to denote the $k$-fold product set of a collection $A$ of group elements, noting that if $A$ is a $K$-approximate group then $|A^{k}| \leq K^{k-1}|A|$ for all positive integers $k$.\vspace{11pt}

\noindent\textsc{Acknowledgments.} EB is supported by the ERC starting grant 208091-GADA.
BG was, while this work was being carried out, a fellow at the Radcliffe Institute at Harvard. He is very happy to thank the Institute for proving excellent working conditions.
TT is supported by a grant from the MacArthur Foundation, by NSF grant DMS-0649473, and by the NSF Waterman award.

The authors are particularly indebted to Tom Sanders for many useful discussions, and to Harald Helfgott for many useful discussions on product expansion estimates, and their relationship with the sum-product phenomenon. Prior work of Helfgott, especially the paper \cite{helfgott-sl3}, has been a key source of inspiration for us. E. Breuillard is particularly indebted to Harald Helfgott for extensive explanations of his prior work during numerous visits to Paris in 2008 and 2009. We also thank Nick Gill for useful conversations and Olivier Guichard for pointing out some inaccuracies in an earlier version of this paper. Finally we acknowledge the intellectual debt we owe to the model-theoretic work of Hrushovski \cite{hrush}, without which we would not have started this project. \vspace{5pt}
\vspace{11pt}

\noindent\textsc{Notes added in proof.}  Prior to this announcement, Nick Gill and Harald Helfgott had announced in conference talks some special cases of our main theorem for $\SL_n(\Z/p\Z)$ and sets $A$ of ``small size'' with respect to $p$. Their results have since been released in \cite{gill}.

Also, simultaneously with the release of this announcement, Pyber and Szabo \cite{pyber} have independently announced a set of results which have significant overlap with those presented here, in particular establishing an alternate proof of Theorem \ref{mainthm1}.  There are some similarities in common in the argument (in particular, in the reliance on Lemma \ref{crucial-lemma-sln}) but the arguments and results are slightly different in other respects.

\section{Statement of results}

In a celebrated paper \cite{helfgott-sl2}, H.~Helfgott provided a satisfactory solution to the classification problem for the group $\SL_2(\F_p)$. His methods adapt easily to (and in fact are rather easier in) $\SL_2(\C)$, which was studied earlier by Chang \cite{chang}.  His arguments give the following result.

\begin{theorem}[Helfgott]\label{helfgott-strong}  Suppose that $A \subseteq \SL_2(k)$ is a $K$-approximate group.
\begin{enumerate}
\item If $k = \C$  then $A$ is $K^C$-controlled by $B$, an \emph{abelian} $K^C$-approximate subgroup of $\SL_2(k)$;
\item If $k = \F_p$ then $A$ is $K^C$-controlled either by a \emph{solvable} $K^C$-approximate subgroup of $\SL_2(k)$ or by $\SL_2(k)$ itself.
\end{enumerate}
\end{theorem}

Helfgott's theorem has found many applications, for example to proving that certain Cayley graphs are expanders \cite{bourgain-gamburd} and to certain nonlinear sieving problems \cite{bgs}. For these applications only the following somewhat weaker statement is necessary.

\begin{theorem}[Helfgott]\label{helfgott-weak}
Suppose that $A \subseteq \SL_2(\F_p)$ is a $K$-approximate group that generates $\SL_2(\F_p)$. Then $A$ is $K^C$-controlled by either $\{id\}$ or by $\SL_2(\F_p)$ itself.
\end{theorem}

Very recently, this result has been extended to fields $\F_q$ of prime power order $q=p^j$ by Dinai \cite{dinai}.  An extension to $\SL_n(\F_p)$ in the case of small $A$ (specifically, $|A| \leq p^{n+1-\delta}$ for some $\delta > 0$) was also recently announced by Gill and Helfgott \cite{gill}.

The proof that such a statement suffices for the sieving work of Bourgain, Gamburd and Sarnak \cite{bgs} is contained in a very recent preprint of Varj\'u \cite{varju}; the original argument of \cite{bgs} required a careful analysis of the \emph{proof} of Helfgott's result.

Our first main result generalises Theorem \ref{helfgott-weak} as follows.

\begin{theorem}\label{mainthm1}
Let $k$ be a finite field. Suppose that $A \subseteq \SL_n(k)$ is a $K$-approximate group that generates $\SL_n(k)$. Then $A$ is $K^{C_n}$-controlled by either $\{\id\}$ or by $\SL_n(k)$ itself.
\end{theorem}

In a much longer second paper \cite{helfgott-sl3}, Helfgott proved this result when $n = 3$, at least in the case when $k = \F_p$ is a prime field. It is known, thanks to work of Bourgain-Gamburd \cite{bourgain-gamburd2} and Varj\'u \cite{varju} respectively, that such a statement suffices in order to fully generalise the aforementioned results on expanders and on the affine sieve. We shall state some of these applications and others in \S \ref{app-sec}.

In proving Theorem \ref{mainthm1} we can quote extensively from Helfgott's paper \cite{helfgott-sl3}. However the structure of the proof seems clearer when one establishes the following result of greater generality.

\begin{theorem}\label{mainthm2}
Let $k$ be a finite field and let $G(k)$ be a Chevalley group\footnote{That is, a group associated to one of the Dynkin diagrams $A_n, B_n, C_n, D_n, E_6, E_7, E_8, F_4$ or $G_2$. It is likely that our results hold in greater generality (for example for all simple groups of Lie type) but we have not yet checked this.}. Suppose that $A \subseteq G(k)$ is a $K$-approximate group that generates $G(k)$. Then $A$ is $K^{C_{\dim(G)}}$-controlled by either $\{\id\}$ or by $G(k)$ itself.
\end{theorem}
Note that the constant $C_{\dim(G)}$ does not depend on the field $k$.

There is a more significant advantage of working in simple groups more general that $\SL_n$. Over $\C$, for example, the general structure theory of algebraic groups implies, roughly speaking, that every Zariski-closed subgroup of $\GL_n(\C)$ admits a quotient by a normal solvable subgroup which is a direct product of simple complex Lie groups. By exploiting this theory, an analogue of Theorem \ref{mainthm1} over $\C$, and the main result of \cite{bg-2}, we are able to establish the following result.

\begin{theorem}\label{mainthm3}
Suppose that $A \subseteq \GL_n(\C)$ is a $K$-approximate subgroup. Then $A$ is $CK^{C_n}$-controlled by $B$, a $CK^{C_n}$-approximate group that generates a nilpotent group of nilpotency class \textup{(}step\textup{)} at most $n-1$.
\end{theorem}

We are also able to say something about the structure of $K$-approximate subgroups of $\GL_n(k)$ where $k$ is a finite field, at least in the case $k = \F_p$, but this seems to be very substantially more difficult and it is not yet clear what the final form of such a result will be.

Qualitative forms of the above theorems follow from the work of Hrushovski \cite{hrush}. The main novelty of our work lies in the polynomial dependence on the approximation parameter $K$, which is absolutely essential for applications. The powers $C_n$ and $\dim(G)$ are in principle explicitly computable. However, if one is willing to sacrifice such information a key portion of the argument (the proof of the Larsen-Pink inequalities, described in \S \ref{larsen-pink-sec}) can be significantly simplified by the use of an ultrafilter argument.

\section{A new proof of Helfgott's $\SL_2$ result}\label{sl2-sec}

To illustrate the main ideas in our paper, we give a reasonably detailed sketch of the proof of Theorem \ref{mainthm1} in the case $n = 2$. The result is due to Helfgott, at least when $k = \F_p$ is a prime field. So are some of the ideas, and indeed we begin by quoting a lemma from \cite{helfgott-sl2}. Recall that a \emph{maximal torus} of $\SL_2(\overline{k})$ a maximal, connected, abelian, diagonalisable algebraic subgroup of $\SL_2(\overline{k})$, or in other words a subgroup conjugate to the group $\{ \left( \begin{smallmatrix} \lambda & 0 \\ 0 & 1/\lambda\end{smallmatrix} \right) : \lambda \neq 0\}$ of diagonal matrices.

Throughout this section, $k$ is a finite field and $A \subseteq \SL_2(k)$ is a $K$-approximate group generating $\SL_2(k)$. We make free use of the notation $\sim$, $\lesssim$, $\gtrsim$ introduced in the introduction, all with reference to this same parameter $K$.

\begin{lemma}[Helfgott]\label{torus-lem}
Let $T \subseteq \SL_2(\overline{k})$ be a maximal torus containing at least one element of $A$ other than the identity. Then
\[ |A^{20} \cap T| \sim |A|^{1/3}.\]
\end{lemma}
The upper bound is \cite[Corollary 5.4]{helfgott-sl3}, whilst the lower bound follows from \cite[Corollary 5.10]{helfgott-sl3} (or rather the proof of it). Helfgott does not explicitly obtain the power $20$ but it is likely that this would follow from his arguments; in any case the exact power is totally unimportant for this sketch.

Helfgott established an analogous result in $\SL_n$ for $n \geq 3$ under the assumption that $T$ contains a \emph{regular semisimple} element of $A$, that is to say a matrix with distinct eigenvalues. Note that a special feature of $\SL_2$ is that all semisimple (diagonalisable) elements other than $\pm \id$ are regular semisimple. One of the main ingredients of our paper is a further generalisation of the upper-bound portion of Lemma \ref{torus-lem} to the case where $\SL_2$ is replaced by an arbitrary simple group of Lie type and $T$ is replaced by an arbitrary algebraic subvariety $V$ of bounded ``complexity''. To do this we adapt an argument of Larsen and Pink \cite{larsen-pink}; the details are sketched in \S \ref{larsen-pink-sec}. In the special case that $V$ is a torus containing a regular semisimple element of $A$ we can obtain a corresponding \emph{lower} bound by placing an upper bound on the intersection of $A^{O(1)}$ with conjugacy classes, much as Helfgott does in \cite{helfgott-sl2,helfgott-sl3}. This is discussed further in \S \ref{larsen-pink-sec}, and in this way we obtain a complete generalisation of Lemma \ref{torus-lem} to arbitrary Chevalley groups.

Returning now to the $\SL_2$ setting, let us say that a torus $T \subseteq \SL_2(k)$ is \emph{involved} with $A$ if $A^2 \cap T$ contains at least one regular semisimple element of $A$ (which, in the case of $\SL_2(k)$, is equivalent to asserting that $A^2 \cap T$ is not contained in $\{ \id, -\id\}$). The crucial observation concerning this notion is as follows.

\begin{lemma}[Conjugation invariance in $\SL_2$]\label{crucial-lemma}
Suppose that $|A| > CK^C$ for a sufficiently large $C$. Then the set of involved tori is invariant under conjugation by all elements of $\SL_2(k)$.
\end{lemma}
\begin{proof} Since $A$ generates $\SL_2(k)$, it suffices to show that if $T$ is involved and if $a \in A$ then $\tilde T := a^{-1}T a$ is also involved.  Let us look at the sets $x T y$ as $x, y$ range over $A^{2}$.
These are all cosets of tori in $\SL_2(\overline{k})$ and, by Lemma \ref{torus-lem}, there are $\lesssim |A|^{4/3}$ of them, since for each fixed $x,y \in A^2$ there are $\gtrsim |A|^{1/3}$ values of $x' \in A^{22}$ with $xT=x'T$ and $\gtrsim |A|^{1/3}$ values of $y' \in A^{22}$ with $Ty = Ty'$, yet the total number of cosets $x'Ty'$ with $x', y' \in A^{22}$ is at most $|A^{22}|^2\lesssim |A|^2$. \emph{A fortiori} there are $\lesssim |A|^{4/3}$ sets of the form $x \tilde T y $ as $x,y$ range over $A$.

By the pigeonhole principle, we can thus find $m \gtrsim |A|^{2/3}$ distinct pairs $(b_i, c_i) \in A \times A$, $i = 1, \dots,m$, such that
\begin{equation}\label{star} b_1 \tilde T c_1 = \dots = b_m \tilde T c_m.\end{equation}
Without loss of generality there are $l\gtrsim |A|^{1/3}$ different values of $b_i$ appearing here, say $b_1,\dots, b_{l}$.  Since each $b_i \tilde T b_i^{-1}$ is a subgroup, it is easy to see that \eqref{star} in fact implies that
\[ b_1 \tilde T b_1^{-1} = b_2 \tilde T b_2^{-1} = \dots = b_{l}\tilde T b_{l}^{-1},\] and hence that $b_1^{-1} b_i$ lies in the normaliser $N_{\SL_2}(\tilde T)$ for all $i = 1,\dots,l$.  But the order of the group $N_{\SL_2}(\tilde T)/\tilde T$ is two (one can easily compute this directly: this group is called the \emph{Weyl group}), and therefore we may find some fixed $i$ and at least $l/2$ values of $j$ such that $(b_1^{-1} b_i)^{-1}(b_1^{-1} b_j) = b_i^{-1} b_j \in \tilde T$. This implies that $|A^2 \cap \tilde T| \gtrsim |A|^{1/3}$ and hence, \emph{a fortiori}, $\tilde T$ is involved. \end{proof}

\emph{Remark.} The arguments here were inspired by the proof that every approximate field is almost a field, a general form of the sum-product theorem.  The idea originated in \cite{bourgain-glibichuk-konyagin} and is also described in \cite[Section 2.8]{tv-book}.  Indeed, if $A$ is a finite subset of a field $k$, say that an element $\xi \in k$ is \emph{involved} with $A$ if the set $A + \xi A$ has cardinality strictly less than $|A|^2$.  If $A$ is approximately closed under addition and multiplication in a certain sense, one can show by arguments not dissimilar to those above that the set of involved elements is finite and is closed under addition and multiplication, i.e. is a genuine finite subfield of $k$. Note that Helfgott used in \cite{helfgott-sl3} an argument not dissimilar to ours when he recast the sum-product phenomenon in terms growth in groups acted upon by an abelian automorphism group. A common feature to all these phenomena is the presence of two distinct, uncorrelated actions (e.g. mutliplication and addition, or left and right translations) whose combination produces growth.

This might also be a good place to mention that, unlike Helfgott in \cite{helfgott-sl2}, we make no appeal to results on the sum-product phenomenon. In fact we are able to deduce the sum-product phenomena from Theorem \ref{mainthm1} in the case $n = 2$ and the so-called Katz-Tao lemma \cite[Lemma 2.53]{tv-book}, thereby providing a different proof of these results.\vspace{11pt}

Given Lemma \ref{crucial-lemma} it is quite an easy matter to complete the proof of Theorem \ref{mainthm1} in the case $n = 2$. Consider the conjugation action of $G = \SL_2(k)$ on the set $X = \{T_1,\dots,T_m\}$ of involved tori (note that these tori are defined over $\overline{k}$, and need not be defined over $k$). By another application of Lemma \ref{torus-lem} and the fact that distinct tori intersect only at $\{\id, -\id\}$ (a special feature of $\SL_2$) it follows that $m \lesssim |A|^{2/3}$.

Thus the orbit of $T_1$ under this conjugation action has size at most $m \sim |A|^{2/3}$, whereas the stabiliser of $T_1$ is $N(T_1) \cap G$, a set of size $\sim |G|^{1/3}$ (by direct verification). It follows from the orbit-stabiliser theorem that
\[ |A|^{2/3}  |G|^{1/3} \gtrsim |G|,\] and therefore $|A| \gtrsim |G|$. This means, of course, that $A$ is $K^C$-controlled by $G$, and the proof is complete.\endproof

\section{Sets that generate $\SL_n$}

Let $k$ be a finite field. Let us see what changes must be made to the argument of the previous section in order to make it work when $n \geq 3$, that is to say to prove Theorem \ref{mainthm1} in general. As previously mentioned, the analogue of Lemma \ref{torus-lem} is the following. Through this section $A \subseteq \SL_n(k)$ is a $K$-approximate group which generates $\SL_n(k)$. All absolute constants, including those implicit in the $\sim$, $\lesssim$ and $\gtrsim$ notation, are allowed to depend on $n$.

\begin{lemma}[Helfgott]\label{torus-lem-sln}
Suppose that $T\subseteq \SL_n(\overline{k})$ is a maximal torus containing at least one regular semisimple element of $A$. Then
\[ |A^{C} \cap T| \sim |A|^{1/(n+1)}.\]
\end{lemma}

Let us now say that a torus $T \subseteq \SL_n(\overline{k})$ is \emph{involved} with $A$ if $A^2 \cap T$ contains at least one regular semisimple element. We have the following direct analogue of Lemma \ref{crucial-lemma}.

\begin{lemma}\label{crucial-lemma-sln}
Suppose that $|A| \geq CK^C$ for sufficiently large $C = C_n$. Then the set of involved tori is invariant under conjugation by the whole group $\SL_n(k)$.
\end{lemma}
\proof The argument is almost exactly the same as that used to prove Lemma \ref{crucial-lemma}. Everything works in precisely the same way, using the fact that the Weyl group $N_G(T)/T$ is finite and bounded in size independently of $k$, up until the very last line where it is shown that $|A^2 \cap \tilde T| \gtrsim |A|^{1/(n+1)}$. It does not follow quite so immediately from this that $\tilde T$ is involved. Writing $\tilde T_{\rss}$ for the set of regular semisimple elements of $\tilde T$, it suffices to show that
\[ |A^2 \cap (\tilde T \setminus \tilde T_{\rss})| \lesssim |A|^{1/(n+1) - \delta}\]
for some $\delta = \delta_n > 0$.
The set $\tilde T \setminus \tilde T_{\rss}$ is contained in a finite union of ``deficient'' subtori of dimension at most $n-2$ (note that the dimension of a maximal torus is $n-1$). Hence the key is to bound the number of points of $A^2$ inside a deficient torus by $\lesssim |A|^{1/(n+1) - \delta}$.

A bound of this type follows from the rather general Larsen-Pink-type inequalities to be presented in the next section, though it may well also follow by modifying Helfgott's approach in \cite[Section 5]{helfgott-sl3} (which does, in any case, resemble the Larsen-Pink argument).\endproof

To finish the proof of Theorem \ref{mainthm1} we proceed much as before, examining the conjugation action of $G = \GL_n(k)$ on the set $X = \{T_1,\dots,T_m\}$ of involved maximal tori. There is one new issue here, which is that it is not immediately clear that there is even \emph{one} involved torus. That this is so follows from an application of the ``escape from subvarieties'' Lemma of Eskin-Mozes-Oh \cite{emo}, much used in Helfgott's papers. This lemma guarantees that, after replacing $A$ by $A^{C}$ if necessary, that a positive proportion of the elements of $A$ are regular semisimple.

To obtain the bound $m \lesssim |A|^{n/(n+1)}$ we need only combine Lemma \ref{torus-lem-sln} with the observation that each regular semisimple element of $A$ lies on a \emph{unique} maximal torus. The rest of the argument proceeds as before.\endproof

\section{A Larsen-Pink-type inequality}\label{larsen-pink-sec}

Suppose that $A \subseteq G(k)$, where $G$ is a Chevalley group. In the last section we saw, in the case $G = \SL_n$, how useful it is to have upper bounds on $|A \cap V|$ for various varieties $V$, specifically maximal tori, deficient tori and conjugacy classes (the latter to get the lower bounds for tori in Lemmas \ref{torus-lem} and \ref{torus-lem-sln}, as described below). It turns out that bounds of this type are available in considerable generality. To obtain them, we adapt an argument of Larsen and Pink \cite{larsen-pink}. The idea of using the Larsen-Pink inequality was inspired by a similar adaptation of that inequality to a model-theoretic setting in \cite{hrush}.

Thus far, we have been dealing only with finite fields $k$. When $k$ is uncountably infinite it makes no sense to say that $A$ generates $G(k)$. It turns out that the correct substitute is the notion of \emph{sufficiently Zariski-dense}. We say that $A \subseteq G(\overline{k})$ is $M$-sufficiently Zariski dense if $A$ is not contained in any proper subvariety  of \emph{complexity at most $M$}, that is to say defined by at most $M$ polynomials of degree at most $M$. The notion of sufficiently Zariski-dense takes a little getting used to. For example, $A = \SL_2(\F_p)$ is $M$-sufficiently Zariski-dense in $\SL_2(\overline{\F}_p)$ when $M \geq M_0(p)$, even though $A$ generates only a very tiny portion of $\SL_2(\overline{\F}_p)$.

\begin{lemma}[Larsen-Pink inequality for approximate groups]\label{lpi}
Let $k$ be any field, suppose that $G(k)$ is a Chevalley group, and let $M$ be a parameter. Then there is an $M_0 = M_0(G,M)$ with the following property. Suppose that $A$ is $M_0$-sufficiently Zariski-dense. Then for any subvariety $V$ of complexity at most $M$ we have
\[ |A \cap V| \leq C |A^C|^{\dim V/\dim G}\]
where $C = C(G,M)$.  In particular, if $A$ is a $K$-approximate group for some $K \geq 1$, then we have
\begin{equation}\label{splash}
 |A^m \cap V| \leq C_m K^{C_m} |A|^{\dim V/\dim G}
\end{equation}
for any $m \geq 1$, where $C_m = C_m(G,M)$.
\end{lemma}

\emph{Sketch proof.}  The claim is clear when $\dim V = 0$ and $\dim V=\dim G$.  Now suppose for sake of contradiction that there exists dimensions $d_-, d_+$ with $0 < d_- \leq d_+ < \dim G$ and varieties $V_-, V_+$ of dimensions $d_-, d_+$ respectively such that the claim fails for $V_-, V_+$, but holds for varieties of dimensions less than $d_-$ or greater than $d_+$.  Using the hypothesis that $G$ is simple and $A$ is sufficiently Zariski-dense, it is possible to conjugate $V_-$ by an element of $A$ so that $V_- \cdot V_+$ has dimension strictly greater than that of $V_+$.  Now, we consider the product map from $(A \cap V_-) \times (A \cap V_+)$ to $(A^2 \cap (V_- \cdot V_+))$.  By hypothesis, the cardinality of $(A^2 \cap (V_- \cdot V_+))$ is at most $C |A^{C}|^{\dim(V_- \cdot V_+)/\dim G} $.  On the other hand, from a further application of the hypothesis, the fibres of this map generically have cardinality at most $C |A^{C}|^{(\dim(V_-)+\dim(V_+)-\dim(V_- \cdot V_+))/\dim G} $.  This implies that $(A \cap V_-) \times (A \cap V_+)$ has cardinality at most $C|A^{C}|^{(\dim(V_-)+\dim(V_+))/\dim(G)}$, which contradicts the construction of $V_-, V_+$. \hfill $\Box$\vspace{11pt}

Several details were suppressed in the above sketch, the most obvious of which were the rather loose use of the word ``generic'' and of the phrase ``sufficiently Zariski-dense''. It is possible to proceed carefully and make rigorous sense of the sketch, but this requires some rather painful quantitative algebraic geometry. Alternatively, an ultrafilter argument may be employed ; whilst this eliminates the need for such quantitative work, it does mean that the constants appearing in the statement of Lemma \ref{lpi} are ineffective and hence so, ultimately, are those in Theorems \ref{mainthm1}, \ref{mainthm2}, \ref{mainthm3} and in the applications described in \S \ref{app-sec}. In the longer paper to follow we present both approaches, though we shall keep our discussion of quantitative algebraic geometry to the minimum required to give a rough shape to the constants appearing in our main theorems.

Let us show how, when $G = \SL_2(\F_p)$ (say), we may use the Larsen-Pink inequality to recover the crucial Lemma \ref{torus-lem} with the constant $20$ replaced by $3$. First of all, we observe that if $A$ generates $\SL_2(\F_p)$ then $A$ is sufficiently Zariski dense. This follows from Dickson's classification of proper subgroups of $\SL_2(\F_p)$. Now if $a \neq \id$ lies in some torus $T$ then, since $a$ is regular semisimple, $T$ is equal to the centraliser $Z(a)$. This being an algebraic subvariety of dimension $1$, the Larsen-Pink inequality immediately gives us the upper bound in Lemma \ref{torus-lem}, namely the inequality
\[ |A^{3} \cap T| \lesssim |A|^{1/3}.\]
To get a bound in the other direction write $C(a)$ for the conjugacy class of $a$ in $\SL_2(\overline{\F}_p)$ and examine the map $\pi : A \rightarrow A^{3} \cap C(a)$ defined by $\pi(x) =x^{-1} a x$. Now $C(a)$ is also an algebraic subvariety, this time of dimension $2$ (given by $\{x : \tr x = \tr a\}$), and thus another application of the Larsen-Pink inequality tells us that
\[ |\im \pi| \lesssim |A|^{2/3}.\]
It follows that some fibre of the map $\pi$ has size $\gtrsim |A|^{1/3}$; but if $x,y$ are two elements in such a fibre we clearly have $xy^{-1} \in Z(a) = T$. Thus we obtain the lower bound
\begin{equation}\label{tor-lower}|A^3 \cap T| \gtrsim |A|^{1/3}\end{equation} as well.
Let us note once again that the preceding argument is very close to one used by Helfgott in \cite{helfgott-sl2} and \cite{helfgott-sl3}.\vspace{11pt}

Let us conclude this section with some remarks on the proof of Theorem \ref{mainthm2}. The argument is exactly the same as the one we sketched for $\SL_2$ in Section \ref{sl2-sec}, except that the required estimates for points on tori and degenerate tori come from the Larsen-Pink inequalities rather than Helfgott's papers and one needs a little algebraic group theory to establish basic properties of the maximal tori of $G(\overline{k})$ and, in particular, the finiteness of the Weyl group $N_G(T)/T$.

\section{Sketch proof of Theorem \ref{mainthm3}.}

In this section we outline, very briefly, the ideas behind the proof of Theorem \ref{mainthm3}. Recall that in this theorem it was claimed that approximate subgroups of $\GL_n(\C)$ are controlled by nilpotent approximate groups.

Let $A \subseteq \GL_n(\C)$ be a $K$-approximate group. The first step of the argument is to show that $A$ is sufficiently Zariski-dense inside \emph{some} algebraic subgroup $G \subseteq \GL_n(\C)$ of bounded complexity, and this is achieved by a dimension descent argument.

If $G$ is simple we are already done by Theorem \ref{mainthm1}: in fact $|A| \leq K^C$. It follows quickly that the same conclusion holds when $G$ is semisimple (has trivial solvable radical), in which case it is  known from the theory of algebraic groups that $G$ is an almost direct product of almost-simple groups. For a general $G$, we may consider the image of $A$ under the projection $G \rightarrow G/\mbox{Rad}(G)$ onto the semisimple part of $G$ which, by the preceding discussion, is bounded in size by $K^C$. This easily implies that $A$ is $K^C$-controlled by a $K^C$-approximate subgroup of $\mbox{Rad}(G)$, a solvable group. Finally, this in turn is $K^{C'}$-controlled by a nilpotent $K^{C'}$-approximate group $B'$ by the main result of \cite{bg-2}. If desired, that same paper could be used to further control $A$ by a \emph{nilpotent progression}.

\section{Applications and further remarks}\label{app-sec}

Some applications of our results regarding the diameter of finite simple groups are mentioned in Helfgott's paper \cite{helfgott-sl2} in the $\SL_2$ case and the proofs adapt straightforwardly to the more general case. For example we get a special case of a conjecture of Babai and Seress (\cite[Conjecture 1.7]{babai}).

\begin{theorem}[Diameter of $G(\F_p)$]
Let $G$ be a Chevalley group. Then there is a constant $C=C(G)>0$ such that the diameter of \emph{every} Cayley graph of $G(\F_p)$ is at most $C\log^C p$.
\end{theorem}

Similarly we obtain a logarithmic bound $O(\log p)$ on the diameter of those Cayley graphs of $G(\F_p)$ that are obtained as reduction mod $p$ of a subset of $G(\Z)$ which generates a Zariski-dense subgroup, using the strong approximation theorem of Matthews-Weisfeiler-Vasserstein \cite{MVW} and the Tits alternative \cite{tits} to get a logarithmic lower bound on the girth of the Cayley graphs. The same logarithmic bound on the diameter also holds for random Cayley graphs of $G(\F_p)$, because the girth of a random Cayley graph of $G(\F_p)$ is at least logarithmic in $p$ by a result of Gamburd et al. \cite{gamburdHSSV}.


Helfgott's results on $\SL_2$ were spectacularly applied by Bourgain and Gamburd \cite{bourgain-gamburd} to show that various families of Cayley graphs on $\SL_2(\F_p)$ are expanders.



They deduced the spectral gap from two ingredients : the classification of approximate subgroups on the one hand (i.e. Helfgott's
theorem for $SL(2)$), and a non-concentration estimate bounding the mass given by the simple random walk on the Cayley graph to every proper subgroup on the other hand.


In subsequent work of Bourgain and Gamburd \cite{bourgain-gamburd2} and \cite{bourgain-gamburd3}, the following theorem is obtained modulo our Theorem \ref{mainthm1} for $\SL_n$. The following, then, is now unconditional.

\begin{theorem}[Quotients of $\SL_n(\Z)$ as expanders]\label{sldexpander}
Let $n \geq 2$. Suppose that $S$ is a finite symmetric subset of $\SL_n(\Z)$ generating a Zariski dense subgroup, and write $S_p$ for the reduction of $S$ modulo $p$. Then the Cayley graphs $\mathcal{G}(\SL_n(\F_p),S_p)$ form a family of expanders as $p \rightarrow \infty$.
\end{theorem}

Using a slightly different method but similar techniques as in the above works of Bourgain and Gamburd \cite{bourgain-gamburd2,bourgain-gamburd3}, in particular random matrix products theory in combination with Nori's theorem, we can prove directly the non-concentration estimate and thus generalize the above to all simple Chevalley groups $G(\Z)$. We thus obtain the following.

\begin{theorem}[Quotients of $G(\Z)$ as expanders]\label{Gexpander}
Let $G$ be a Chevalley group. Suppose that $S$ is a finite symmetric subset of $G(\Z)$ generating a Zariski dense subgroup, and write $S_p$ for the reduction of $S$ modulo $p$. Then the Cayley graphs $\mathcal{G}(G(\F_p),S_p)$ form a family of expanders as $p \rightarrow \infty$.
\end{theorem}

Finally let us mention once again that as a consequence of our results and the paper of Varj\'u \cite{varju} the ``affine sieve''  of Bourgain, Gamburd and Sarnak \cite{bgs} may be applied in much more general contexts. We refer the reader to Varj\'u's paper for details.

\setcounter{tocdepth}{1}


\end{document}